\documentclass[11pt]{article}
\usepackage[utf8x]{inputenc}
\usepackage{quoting}
\newcommand{\vertiii}[1]{{\left\vert\kern-0.25ex\left\vert\kern-0.25ex\left\vert #1
    \right\vert\kern-0.25ex\right\vert\kern-0.25ex\right\vert}}
\quotingsetup{vskip=0pt,leftmargin=10pt,rightmargin=10pt}

\def\showauthornotes{1}
\def\showkeys{0}
\def\showdraftbox{0}
\def\usemicrotype{1}
\def\showfixme{0}

\usepackage{etex}

\usepackage[l2tabu, orthodox]{nag}

\usepackage{xspace,enumerate}

\usepackage[dvipsnames]{xcolor}

\usepackage[T1]{fontenc}
\usepackage[full]{textcomp}

\usepackage[american]{babel}

\usepackage{mathtools}

\usepackage{amsthm}

\newtheorem{theorem}{Theorem}[section]
\newtheorem*{theorem*}{Theorem}

\newtheorem{proposition}[theorem]{Proposition}
\newtheorem*{proposition*}{Proposition}
\newtheorem{lemma}[theorem]{Lemma}
\newtheorem*{lemma*}{Lemma}

\newtheorem*{conjecture*}{Conjecture}
\newtheorem{fact}[theorem]{Fact}
\newtheorem*{fact*}{Fact}

\newtheorem*{hypothesis*}{Hypothesis}

\theoremstyle{definition}

\theoremstyle{remark}

\newtheorem*{claim*}{Claim}

\newtheorem*{remark*}{Remark}

\newtheorem*{observation*}{Observation}

\usepackage[
letterpaper,
top=0.7in,
bottom=0.9in,
left=1in,
right=1in]{geometry}

\usepackage[varg]{pxfonts}
\usepackage{textcomp}
\usepackage[scr=rsfso]{mathalfa}% \mathscr is fancier than \mathcal
\usepackage{bm} % load after all math to give access to bold math
\let\mathbb\varmathbb

\ifnum\showkeys=1
\usepackage[color]{showkeys}
\fi

\usepackage{hyperref}
\hypersetup{
pagebackref,
colorlinks=true,
urlcolor=blue,
linkcolor=blue,
citecolor=OliveGreen,
}
\usepackage{prettyref}

\newcommand{\savehyperref}[2]{\texorpdfstring{\hyperref[#1]{#2}}{#2}}

\newrefformat{eq}{\savehyperref{#1}{\textup{(\ref*{#1})}}}
\newrefformat{lem}{\savehyperref{#1}{Lemma~\ref*{#1}}}
\newrefformat{def}{\savehyperref{#1}{Definition~\ref*{#1}}}
\newrefformat{thm}{\savehyperref{#1}{Theorem~\ref*{#1}}}
\newrefformat{cor}{\savehyperref{#1}{Corollary~\ref*{#1}}}
\newrefformat{cha}{\savehyperref{#1}{Chapter~\ref*{#1}}}
\newrefformat{sec}{\savehyperref{#1}{Section~\ref*{#1}}}
\newrefformat{app}{\savehyperref{#1}{Appendix~\ref*{#1}}}
\newrefformat{tab}{\savehyperref{#1}{Table~\ref*{#1}}}
\newrefformat{fig}{\savehyperref{#1}{Figure~\ref*{#1}}}
\newrefformat{hyp}{\savehyperref{#1}{Hypothesis~\ref*{#1}}}
\newrefformat{alg}{\savehyperref{#1}{Algorithm~\ref*{#1}}}
\newrefformat{rem}{\savehyperref{#1}{Remark~\ref*{#1}}}
\newrefformat{item}{\savehyperref{#1}{Item~\ref*{#1}}}
\newrefformat{step}{\savehyperref{#1}{step~\ref*{#1}}}
\newrefformat{conj}{\savehyperref{#1}{Conjecture~\ref*{#1}}}
\newrefformat{fact}{\savehyperref{#1}{Fact~\ref*{#1}}}
\newrefformat{prop}{\savehyperref{#1}{Proposition~\ref*{#1}}}
\newrefformat{prob}{\savehyperref{#1}{Problem~\ref*{#1}}}
\newrefformat{claim}{\savehyperref{#1}{Claim~\ref*{#1}}}
\newrefformat{relax}{\savehyperref{#1}{Relaxation~\ref*{#1}}}
\newrefformat{red}{\savehyperref{#1}{Reduction~\ref*{#1}}}
\newrefformat{part}{\savehyperref{#1}{Part~\ref*{#1}}}

\newcommand{\Sref}[1]{\hyperref[#1]{\S\ref*{#1}}}

\usepackage{nicefrac}
\ifnum\usemicrotype=1
\usepackage{microtype}
\fi

\ifnum\showauthornotes=1
\newcommand{\Authornote}[2]{{\sffamily\small\color{red}{[#1: #2]}}}
\newcommand{\Authornotecolored}[3]{{\sffamily\small\color{#1}{[#2: #3]}}}
\newcommand{\Authorcomment}[2]{{\sffamily\small\color{gray}{[#1: #2]}}}
\newcommand{\Authorstartcomment}[1]{\sffamily\small\color{gray}[#1: }

\newcommand{\Authorfnote}[2]{\footnote{\color{red}{#1: #2}}}
\newcommand{\Authorfixme}[1]{\Authornote{#1}{\textbf{??}}}
\newcommand{\Authormarginmark}[1]{\marginpar{\textcolor{red}{\fbox{\Large #1:!}}}}
\else
\newcommand{\Authornote}[2]{}
\newcommand{\Authornotecolored}[3]{}
\newcommand{\Authorcomment}[2]{}
\newcommand{\Authorstartcomment}[1]{}

\newcommand{\Authorfnote}[2]{}
\newcommand{\Authorfixme}[1]{}
\newcommand{\Authormarginmark}[1]{}
\fi

\ifnum\showfixme=0

\fi

\usepackage{boxedminipage}

\newcommand\bdot\bullet

\ifx\mathds\undefined % use double stroke fonts if available

\else

\fi

\DeclareMathOperator{\Tr}{Tr}
\renewcommand{\leq}{\leqslant}

\renewcommand{\geq}{\geqslant}

\ifnum\showdraftbox=1
\newcommand{\draftbox}{\begin{center}
  \fbox{%
    \begin{minipage}{2in}%
      \begin{center}%
          \Large\textsc{Working Draft}\\%
        Please do not distribute%
      \end{center}%
    \end{minipage}%
  }%
\end{center}
\vspace{0.2cm}}
\else
\newcommand{\draftbox}{}
\fi

\let\epsilon=\varepsilon

\numberwithin{equation}{section}

\newcommand\MYcurrentlabel{xxx}

\newcommand{\MYstore}[2]{%
  \global\expandafter \def \csname MYMEMORY #1 \endcsname{#2}%
}

\newcommand{\MYload}[1]{%
  \csname MYMEMORY #1 \endcsname%
}

\newcommand{\MYnewlabel}[1]{%
  \renewcommand\MYcurrentlabel{#1}%
  \MYoldlabel{#1}%
}

\newcommand{\MYdummylabel}[1]{}

\newcommand{\torestate}[1]{%
  \let\MYoldlabel\label%
  \let\label\MYnewlabel%
  #1%
  \MYstore{\MYcurrentlabel}{#1}%
  \let\label\MYoldlabel%
}

\newcommand{\restatetheorem}[1]{%
  \let\MYoldlabel\label
  \let\label\MYdummylabel
  \begin{theorem*}[Restatement of \prettyref{#1}]
    \MYload{#1}
  \end{theorem*}
  \let\label\MYoldlabel
}

\newcommand{\restatelemma}[1]{%
  \let\MYoldlabel\label
  \let\label\MYdummylabel
  \begin{lemma*}[Restatement of \prettyref{#1}]
    \MYload{#1}
  \end{lemma*}
  \let\label\MYoldlabel
}

\newcommand{\restateprop}[1]{%
  \let\MYoldlabel\label
  \let\label\MYdummylabel
  \begin{proposition*}[Restatement of \prettyref{#1}]
    \MYload{#1}
  \end{proposition*}
  \let\label\MYoldlabel
}

\newcommand{\restatefact}[1]{%
  \let\MYoldlabel\label
  \let\label\MYdummylabel
  \begin{fact*}[Restatement of \prettyref{#1}]
    \MYload{#1}
  \end{fact*}
  \let\label\MYoldlabel
}

\newcommand{\restate}[1]{%
  \let\MYoldlabel\label
  \let\label\MYdummylabel
  \MYload{#1}
  \let\label\MYoldlabel
}

\newcommand{\addreferencesection}{
  \phantomsection
  \addcontentsline{toc}{section}{References}
}

\let\origparagraph\paragraph
\renewcommand{\paragraph}[1]{\origparagraph{#1.}}

\allowdisplaybreaks

\usepackage{paralist}

\usepackage{comment}

\let\citet\cite

\theoremstyle{definition}
\usepackage{relsize}
\usepackage[font=footnotesize]{caption}
\usepackage{yfonts}

\DeclareUrlCommand\email{}

\usepackage{appendix}
\usepackage{amssymb}
\usepackage{amsfonts}
\usepackage{latexsym}
\usepackage{amsmath}

\usepackage[T1]{fontenc}
\usepackage{fullpage,appendix}
\usepackage{dsfont}
\usepackage[capitalize]{cleveref}

\usepackage{color}
\usepackage{tikz}

\newcommand{\ignore}[1]{}
\definecolor{corlinks}{RGB}{64,128,128}
\definecolor{cormenu}{RGB}{0,37,94}
\definecolor{corurl}{RGB}{0,46,91}

\renewcommand{\S}{\mathbb{S}}

\renewcommand{\int}{\mathsf{int}}

\renewcommand{\hat}{\widehat}

\makeatletter
\let\orgdescriptionlabel\descriptionlabel
\renewcommand*{\descriptionlabel}[1]{%
  \let\orglabel\label
  \let\label\@gobble
  \phantomsection
  \edef\@currentlabel{#1}%
  \let\label\orglabel
  \orgdescriptionlabel{#1}%
}
\makeatother

\title{A Matrix Bernstein Inequality for Strong Rayleigh Distributions\footnote{This paper was submitted to STOC 2020, during the process of which another paper \cite{ABY19} appeared which proves a slightly stronger result. This note is being released in its unedited form.}}

\author{
Tarun Kathuria \\ EECS, UC Berkeley \\
\text{tarunkathuria@berkeley.edu}
}

\date{}

\begin{document}

\maketitle
\draftbox
\thispagestyle{empty}

\begin{abstract}
The Entropy method provides a powerful framework for proving scalar concentration inequalities by establishing functional inequalities like Poincare and log-Sobolev inequalities. These inequalities are especially useful for deriving concentration for dependent distributions coming from stationary distributions of Markov chains. In contrast to the scalar case, the study of matrix concentration inequalities is fairly recent and there does not seem to be a general framework for proving matrix concentration inequalities using the Entropy method. In this paper, we initiate the study of Matrix valued functional inequalities and show how matrix valued functional inequalities can be used to establish a kind of iterative variance comparison argument to prove Matrix Bernstein like inequalities. Strong Rayleigh distributions are a class of distributions which satisfy a strong form of negative dependence properties. We show that scalar Poincare inequalities recently established for a flip-swap markov chain considered in the work of Hermon and Salez \cite{HS19} can be \textit{lifted} to prove matrix valued Poincare inequalities. This can then be combined with the iterative variance argument to prove a version of matrix Bernstein inequality for Lipshitz matrix-valued functions for Strong Rayleigh measures. 
\end{abstract}

\setcounter{page}{1}
\color{black}
\section{Introduction}
Concentration inequalities bound tail probabilities of general functions of random variables. Several methods have been known to prove such inequalities, including martingale methods, information theoretic methods, Talagrand's induction method as well as a variety of problem-specific methods. Ledoux (see \cite{Led} for a survey) developed a powerful framework known as the "entropy method" for proving concentration inequalities. This framework is concerned with proving functional inequalities (like Poincare and log-Sobolev inequalities) which bound some global divergence of any function whose concentration properties are of interest to us by quantities which measure "local variation" of this function under some diffusion process whose stationary measure is our measure of interest. This framework can be used to recover strong forms of concentration results for scalar functions which seem to be inaccessible through many other strategies for proving concentration. One prominent example of this is concentration of scalar Lipshitz functionals of distributions which are the stationary measure of some rapid mixing markov chain. We refer to the excellent book \cite{BLM13} for details on this and other methods for proving scalar concentration inequalities. The entropy method has also been immensely useful in proving deep theorems in the analysis of boolean functions and we refer the reader to \cite{OD} for more details on this fascinating topic.

The field of Matrix valued concentration of measure is concerned with the phenomenon of matrix valued functions being close (usually in spectral norm) to their expected value with high probability. In comparison to scalar concentration inequalities, results on Matrix concentration are much more recent. However, they have already had tremendous applications in the field of computer science, statistics and information theory. Some examples include applications to spectral graph theory, statistical learning problems, deep learning, sketching . While a large class of applications can be handled using matrix concentration of sums of independent random matrices, there are several applications where there is some dependence structure in the sums of random matrices. Some examples of such applications are simplifying Laplacian linear equations (cite), semi-streaming graph sparsification among others.

Despite the recent flurry of activity in proving and using matrix concentration inequalities, as pointed out by \cite{CT13}, one frequently encounters random matrices that do not submit to existing techniques and the study of matrix concentration inequalities remains an active area of investigation.

Negative dependence of random variables is an intuitive property that suggests that it might help with concentration of measure. A class of negatively dependent measures are strong Rayleigh measures which were introduced and deeply studied in the work of \cite{BBL09}. Examples of such distributions include determinantal measures and random spanning tree distributions. \cite{BBL09} show that strong Rayleigh distributions satisfy the strongest form of negative dependence properties. These negative dependence properties were recently exploited
to design approximation algorithms  (cite) They have also been used in numerical linear algebra and algorithmic fairness for selecting subsets of elements which maximize some notion of diversity.

In this paper, we are interested in understanding matrix concentration of random matrices whose distribution is coming from the stationary distribution of a Markov chain. While some results in this line of work exist in the work of \cite{PMT,MJCFT} which extends Chaterjee's work \cite{Chaterjee} on proving scalar concentration inequalities via Stein's method. The kind of markov chains that this line of work can handle is fairly restricted as it requires a stringent condition known as the Dobrushin condition which is not always to prove. Nevertheless, these results imply strong matrix concentration results in a variety of scenarios, most notably matrix Efron Stein inequalities. Another line of work is the paper of Chen and Tropp \cite{CT13} where they are  interested in extending the entropy method to the matrix valued case. However, their results are fairly restrictive since they require rotational symmetry of the random matrices. The reason for the rotational symmetry requirement is the need for a "Herbst argument" which requires building a differential equation for the exponential moment generating function. However, since \cite{CT13} can only control trace of quantities similar to variance and Dirichlet form (which we also study here). Since they are only controlling the trace, they can only transform the trace of those functions into a differential form assuming rotational symmetry of those matrices which lets them ignore any non-commutative structure of the sums of those random matrices. We refer the reader to \cite{CT13} for a more detailed discussion. Our main idea here is to construct matrix valued functional inequalities and show how these can be used to derive matrix concentration inequalities without any rotational symmetry assumption. While we cannot prove matrix valued (modified) log-Sobolev inequalities we do show that matrix concentration by establishing matrix valued Poincare inequalities (which are far more common for various markov chains as opposed to modified log-sobolev). While our techniques seem to be fairly general, we illustrate it for deriving matrix concentration for strongly Rayleigh measures.

We now formally define strong Rayleigh measures. Let $\mu : 2^{[n]}\rightarrow \mathbb{R}_{+}$ be a probability distribution on the subsets of $[n]=\{1,2,\ldots,n\}$. In particular, we assume that $\mu(\cdot)$ is nonnegative and $\sum\limits_{S \subseteq [n]}\mu(S)=1$. We associate a multi-affine polynomial with variables $z_1,\ldots,z_n$ to $\mu$
\begin{align*}
  g_\mu(z_1,\ldots,z_n)=\sum\limits_{S \subseteq [n]}\mu(S)\prod_{i\in S}z_i
\end{align*}
The polynomial $g_\mu$ is known as the generating polynomial of $\mu$. We say $\mu$ is $k$-homogenous if $g_\mu$ is a homegenous polynomial of degree $k$. A polynomial $p(z_1,\ldots,z_n)\in \mathbb{C}[z_1,\ldots,z_n]$ is said to be stable if whenever Im$(z_i)>0$ for all $i\in[n]$, $p(z_1,\ldots,z_n)\neq 0$. We say $p$ is real stable is $p$ is stable and all its coefficients are real.
Anari et al. \cite{AGR16} proved a mixing time bound by proving a poincare constant for a bases exchange markov chain associated to strong Rayleigh measures. Hermon and Salez\cite{HS19} gave a unified proof of a log-sobolev and Poincare inequality in the scalar case. One of our main results is showing that one can effortlessly follow their exact proof and make minor modifications(using operator convexity of certain functions) to prove a matrix valued Poincare inequality (which has also been considered in \cite{CHT15,CH15,CH16}). Our second main result implies that one can use matrix valued Poincare inequalities to prove matrix Bernstein like concentration inequalities. See \cite{Led} for how to prove similar concentration inequalities for scalar lipshitz valued functions using Poincare and log-sobolev inequalities. Combining our two main theorems, we get the following result.

\begin{theorem}\label{thm:srBern} Given a $k$-homogenous strong Rayleigh measure $\pi$ supported on $n$ elements and given a L-Lipshitz matrix valued function $\mathbf{F}:\{0,1\}^n\rightarrow\mathbb{S}_d$ and let $\mathbf{\mu}=\mathbb{E}_{x\sim\pi}[\mathbf{F(x)}]$. We have for all $t>0$,
  \begin{align*}
    \mathsf{Pr}_{x \sim \pi}[\|\mathbf{F(x)-\mathbf{\mu}}\|\geq t]\leq 2 de^{-t^2/32(kL^2 + t\sqrt{k}L)}
  \end{align*}
\end{theorem}
A matrix valued function $\mathbf{F}:\Omega\rightarrow \mathbb{S}_d$ is L-Lipshitz if for all $x,y \in \Omega$
\begin{align*}
  \|\mathbf{F(x)}-\mathbf{F(y)}\|\leq L\cdot d(x,y)
\end{align*}
with respect to some distance function between states $d(x,y)$. In this paper, we will use the hamming distance.
We remark that a matrix Chernoff bound for strong Rayleigh measures was proven by \cite{KyngSong18} by martingale arguments inspired by the scalar concentration case by \cite{PP14}. Their application was to prove that spanning trees can be used to construct spectral sparsifiers. While our concentration result can't prove a similar statement, our result can be used in situations theirs cant. We compare our result with theirs in Section \ref{sec:ks}.

We finally define some terminology we will require pertaining to Markov chains and refer the reader to \cite{MT06,LPW} for more details.  Let the stationary distribution of a reversible Markov chain with state space $\Omega$ be $\mu$
\begin{align*}
  \mathbb{E}_\mu[\mathbf{F}]&= \sum\limits_{x \in \Omega}\mu(x)\mathbf{F(x)}\\
  \mathbf{Var}_\mu[\mathbf{F}]&=\mathbb{E}_\mu[\mathbf{F}^2]-(\mathbb{E}(\mathbf{F(x)})^2
\end{align*}
Consider a generator for the Markov chain $Q\in\mathbb{R}^{\Omega\times\Omega}$ with non-negative off-diagonal entries with each row summing up to 0 and satisfying the local balance equations
\begin{align*}
  \pi(x)Q(x,y) = \pi(y)Q(y,x)
\end{align*}
For \textit{normalized} generators, one can associate a stochastic matrix $P$ such that $Q=P-I$. In such cases, $P$ also satsifies local balance. Finally, the Dirchlet form for scalar functions is $\mathcal{E}_Q[f,g]=\sum_{x\in\Omega,y\in\Omega}\pi(x)Q(x,y)(f(x)-f(y))^2$. It is easy to see that $\mathcal{E}_Q[f,g]=\mathcal{E}_P[f,g]$. For matrix valued functions $\mathbf{F,G}:\Omega\rightarrow \mathbb{S}_d$, the dirchlet form is defined as $\mathcal{E}_Q[\mathbf{F},\mathbf{G}]=\mathcal{E}_P[\mathbf{F},\mathbf{G}]=\sum_{x\in\Omega,y\in\Omega}\pi(x)Q(x,y)(\mathbf{F(x)}-\mathbf{F(y)})^2$
where $\mathbb{S}_d$ is the space of symmetric matrices of size $d\times d$.

\section{Matrix Valued Poincare Inequality for Strong Rayleigh measures}
In this section, we prove our matrix valued Poincare inequality for measures satisfying the stochastic covering property (SCP). We first need some notation in order to define SCP measures. Let $\{e_i\}_{i=1}^{n}$ be the coordinate basis of $\mathbb{R}^n$. For $x,y \in \{ 0,1\}^n$, we write $x\triangleright y$ if $x$ can be obtained by $y$ by increasing at most one coordinate from 0 to 1, i.e.,
\begin{align*}
  x=y \ \text{ or } \ \exists i \in [n], \ x=y+e_i
\end{align*}
This \textit{covering} relation can be lifted to measures on $\{ 0,1\}^n$ by writing $\pi \triangleright \pi'$ whenever there is a coupling of $\pi$ and $\pi'$ which is supported on the set $\{(x,y): x \triangleright y\}$. We say that $\pi$ has the SCP if for any $S \subseteq [n]$ and $x,y \in \{0,1\}^S$, we have
\begin{align*}
  x \triangleright y \implies \pi(X_{S^c} = \cdot | X_S = y)\triangleright \pi(X_{S^c} = \cdot | X_S = x)
\end{align*}
This property was proposed by Pemantle and Peres \cite{PP14} and they proved that all strong Rayleigh measures satisfy the SCP. However, there are examples of matroids for which the uniform distribution has the SCP but do not have the strong Rayleigh property. Hermon and Salez \cite{HS19} consider the following local dynamics :
\begin{align*}
  x \sim y \iff x,y \text{ differ by a flip or a swap}
\end{align*}
where $x,y$ differ by a flip if $x=y \pm e_i$ for some $i$ and by a swap if $y=x+e_j-e_k$ for distinct $j$ and $k$. And then they say that $Q$ is a flip-swap walk Markov generator for $\pi$ which is reversible under $\pi$ and we have
\begin{align*}
  Q_{xy} > 0 \iff x\sim y
\end{align*}
If in addition, we have $\Delta(Q) = \max\{-Q_{xx} : x \in \Omega\} \leq 1$, we say that $Q$ is normalized and $Q$ may hence be written as $Q=P-I$ for some stochastic matrix $P$. \cite{HS19} proved that there is a normalized flip-swap walk for an SCP distribution has a modified log-Sobolev and Poincare constant $1/2k$. We will prove using their proof strategy that this walk actually has it's Matrix valued Poincare constant also $1/2k$.
\begin{theorem}\label{thm:mpc} If $\pi$ is a $k$-homogenous SCP measure, then there is a  normalized flip-swalk walk $Q$ for $\pi$ such that for any matrix valued function $\mathbf{f} : \Omega \rightarrow \mathbb{S}_d$
  \begin{align*}
    \lambda\mathbf{Var}[\mathbf{f}]\preceq \mathcal{E}(\mathbf{f},\mathbf{f})
  \end{align*}
  with $\lambda \geq \frac{1}{2k}$. Furthermore, the conclusion remains valid without the homogeneity assumption, with $k=n/2$
\end{theorem}
The proof of the scalar equivalent of the above statement in \cite{HS19} follows the chain decomposition framework of \cite{JSTV04} which we elaborate on next. Let $Q$ be any reversible Markov generator with respect to some probability distribution $\pi$ on a finite set $\Omega$. Consider any partition of the state space :
\begin{align*}
  \Omega = \bigcup\limits_{i \in \mathcal{I}}\Omega_i
\end{align*}
We define two chains using this partition. Let the \textit{projection chain} induced by this partition be the Markov chain whose state space is $\mathcal{I}$ and whose Markov generator $\hat{Q}$ is defined as follows: for any $i,j \in \Omega$,
\begin{align*}
  \hat{Q}(i,j) = \frac{1}{\hat{\pi}(i)}\sum\limits_{x \in \Omega_i, y \in \Omega_j} \pi(x) Q(x,y) \ \ \text{ where } \hat{\pi}(i) = \sum\limits_{x \in \Omega_i} \pi(x)
\end{align*}
Note that $\hat{\pi}$ is a probability distribution on $\mathcal{I}$ and that it is reversible under $\hat{Q}$,i.e.,
\begin{align*}
  \hat{\pi}(i) \hat{Q}(i,j) =   \hat{\pi}(j) \hat{Q}(j,i)
\end{align*}
We also have, that for each $i \in \mathcal{I}$, the \textit{restriction chain} on $\Omega_i$ is the Markov chain whose state space is $\Omega_i$ and whose Markov geneator $Q_i$ is defined as: for any states $x,y \in \Omega_i$,
\begin{align*}
  Q_i(x,y) = Q(x,y)
  \end{align*}
with the diagonal entries being adjusted so that the rows of $Q_i$ sum to zero. Clearly, then we have \begin{align*}\hat{\pi}_i(x) = \frac{\pi(x)}{\hat{\pi}(i)}\end{align*}
which is reversible under $Q_i$. Now suppose that for each $(i,j) \in \mathcal{I}^2$ with $\hat{Q}_{ij} > 0$, we are given a coupling $\kappa_{ij} : \Omega_i \times \Omega_j \rightarrow [0,1]$ of the probability distribution $\pi_i$ and $\pi_j$,i.e.,
\begin{align*}
  \forall x \in \Omega_i, \ \sum\limits_{y \in \Omega_j}\kappa_{ij}(x,y) = \pi_i(x)\\
  \forall y \in \Omega_j, \ \sum\limits_{x \in \Omega_i}\kappa_{ij}(x,y) = \pi_j(y)
\end{align*}
We will also need the following measure of quality of the coupling
\begin{align*}
  \chi = \min\bigg\{\frac{\pi(x)Q(x,y)}{\hat{\pi}(i)\hat{Q}(i,j)\kappa_{ij}(x,y)}\bigg\}
\end{align*}
where the minimum runs over all $(x,y,i,j)$ such that the denominator is positive.

The main idea of \cite{JSTV04} is to relate the Poincare and (modified) log-Sobolev constants of the original chain to the projection and restriction chains. Given the above setup, we would now like to prove recursive functional inequalities over the restriction and projection chains. Inequalities of this form are the workhorse of all applications of the induction/chain decomposition idea \cite{JSTV04}. Our next lemma establishes that scalar recursive Poincare inequalities can be lifted to matrix Poincare inequalities, using operator convexity properties of the square function. We first need the following proposition.
\begin{proposition}\label{prop:diffOC} The function $f(\mathbf{X},\mathbf{Y}) = (\mathbf{X}-\mathbf{Y})^2$ is jointly operator convex.
\end{proposition}
\begin{proof}
For any $t\in [0,1]$, $\mathbf{X}_1,\mathbf{X}_2,\mathbf{Y}_1,\mathbf{Y}_2$, we have
  \begin{align*}
    \big( (t \mathbf{X}_1 + (1-t)\mathbf{X}_2) - (t \mathbf{Y}_1 + (1-t)\mathbf{Y}_2)\big)^2&=\big( t (\mathbf{X}_1-\mathbf{Y}_1) + (1-t)(\mathbf{X}_2 -\mathbf{Y}_2)\big)^2\\
    &\preceq t(\mathbf{X}_1-\mathbf{Y}_1)^2+(1-t)(\mathbf{X}_1-\mathbf{Y}_1)^2
  \end{align*}
  where the last step follows from the operator convexity of the square function applied to $\mathbf{X}_1-\mathbf{Y}_1$ and $\mathbf{X}_2-\mathbf{Y}_2$
\end{proof}

\begin{lemma}{[Recursive Poincare inequalities]}\label{lem:recPI} With the above notation, we have
  \begin{align*}\lambda(Q) \geq \min\big\{\chi\lambda(\hat{Q}),\min\limits_{i \in \mathcal{I}}\lambda(Q_i)\big\}
  \end{align*}
  where the $\lambda(\cdot)$ represent the constant in the matrix valued Poincare inequality for the associated generator.
\end{lemma}
\begin{proof}
  Let $\mathbf{F}:\Omega \rightarrow \mathbb{S}_d$. The Dirichlet form can be written as
  \begin{align*}
    \mathcal{E}(\mathbf{F},\mathbf{F}) = \frac{1}{2}\sum\limits_{x,y\in\Omega}\pi(x)Q(x,y)(\mathbf{F(x)}-\mathbf{F(y)})^2
\end{align*}
We can also define the projection of a function $\hat{\mathbf{F}}:\mathcal{I}\rightarrow \mathbb{S}_d$ by $\hat{\mathbf{F}}(i)=\mathbb{E}_{\pi_i}[\mathbf{F}]$. Hence, we can also decompose the variance and Dirichlet form of a matrix valued function as follows:
\begin{align*}
  \mathbf{Var}_\pi[\mathbf{F}]&=\sum\limits_{i \in \mathcal{I}}\hat{\pi}(i)\mathbf{Var}_{\pi_i}[\mathbf{F}] + \mathbf{Var}_{\hat{\pi}}[\hat{\mathbf{F}}]\\
  \mathcal{E}_\pi(\mathbf{F},\mathbf{F}) &= \sum\limits_{i \in \mathcal{I}}\hat{\pi}(i)\mathcal{E}_{\pi_i}(\mathbf{F},\mathbf{F})+\frac{1}{2}\sum\limits_{i \neq j}\sum\limits_{x\in\Omega_i, y \in \Omega_j}\pi(x)Q(x,y)(\mathbf{F(x)}-\mathbf{F(y)})^2
\end{align*}
We would now like to make a term-by-term comparison between the right hand sides of the above two equations. In order to do this, we would like to replace the second term in the Dirichlet form equation by the Dirichlet form for $\mathcal{E}_{\hat{\pi}}(\mathbf{\hat{F}},\mathbf{\hat{F}})$. This is where the coupling is useful. Consider any pair $(i,j)\in\mathcal{I}$ such that $\hat{Q})i,j$>0
. Since $\kappa_{ij}$ is a coupling of $\pi_i$ and $\pi_j$, we have
\begin{align*}
  \mathbf{\hat{F}}(i) &= \sum\limits_{x\in\Omega_i,y\in\Omega_j}\kappa_{ij}(x,y)\mathbf{F}(x)\\
  \mathbf{\hat{F}}(j) &= \sum\limits_{x\in\Omega_i,y\in\Omega_j}\kappa_{ij}(x,y)\mathbf{F}(y)\\
\end{align*}
We can now use Lemma \ref{prop:diffOC} and operator Jensen's inequality (Lemma \ref{lem:jensen}) to conclude that
\begin{align*}
  \sum\limits_{x\in\Omega_i, y\in\Omega_j}\kappa_{ij}(x,y)(\mathbf{F(x)}-\mathbf{F(y)})^2 \succeq (\mathbf{\hat{F}(i)}-\mathbf{\hat{F}(j)})
\end{align*}
Multiplying the above equation by $\chi\hat{\pi}(i)\hat{Q}(i,j)$ we have
\begin{align*}
  \sum\limits_{x\in\Omega_i,y\in\Omega_j}\pi(x)Q(x,y)(\mathbf{F(x)}-\mathbf{F(y)})^2 \succeq \chi \hat{\pi}(i)\hat{Q}(i,j)(\mathbf{\hat{F}(i)}-\mathbf{\hat{F}(j)})^2
\end{align*}
It is easy to see that the above equation also holds when $\hat{Q}(i,j)=0$ rather than just when $\hat{Q}(i,j)=0$. Now, summing up over all distinct $i,j\in \mathcal{I}$, we get that
\begin{align*}
  \frac{1}{2}\sum\limits_{i \neq j}\sum\limits_{x\in\Omega_i, y \in \Omega_j}\pi(x)Q(x,y)(\mathbf{F(x)}-\mathbf{F(y)})^2&\succeq\chi \mathcal{E}_{\hat{\pi}}[\mathbf{\hat{F}},\mathbf{\hat{F}}]\\
  \implies \mathcal{E}_\pi[\mathbf{F},\mathbf{F}]\succeq \sum\limits_{i \in \mathcal{I}}\hat{\pi}(i)\mathcal{E}_{\pi_i}[\mathbf{F},\mathbf{F}]+\chi\mathcal{E}_{\hat{\pi}}[\mathbf{\hat{F}},\mathbf{\hat{F}}]
  \end{align*}
Comparing this with the Variance decomposition, we get
  \begin{align*}\begin{cases}
  \mathcal{E_{\hat{\pi}}}[\mathbf{\hat{F}},\mathbf{\hat{F}}]&\succeq \hat{\kappa}\mathbf{Var}_{\hat{\pi}}[\mathbf{\hat{F]}}\\
  \mathcal{E}_{\pi_i}[\mathbf{F},\mathbf{F}]&\succeq \kappa_{i}\mathbf{Var}_{\pi_{i}}[\mathbf{F}]\forall i\in\mathcal{I},
  \end{cases}\implies\mathcal{E}_{\pi}[\mathbf{F}]\succeq\min\{\chi\hat{\kappa},\min_{i\in\mathcal{{I}}}\kappa_{i}\}\mathbf{Var}_{\pi}[\mathbf{F}]
\end{align*}
Since we did not assume anything about our function $\mathbf{F}$, we get the desired conclusion.
\end{proof}

We will also need the following estimate on the quality of the coupling. Since the proof does not depend on the function (and hence there is no difference in the scalar and matrix valued case), we present the lemma without proof.
\begin{lemma}{[Crude lower-bound on $\chi$, Lemma 2 in \cite{HS19}]} We always have
  \begin{align*}
    \frac{\pi(x) Q(x,y)}{\hat{\pi}(i) \hat{Q}(i,j)\kappa_{ij}(x,y)}\geq \max\bigg\{\frac{Q(x,y)}{\hat{Q}(i,j)},\frac{Q(y,x)}{\hat{Q}(j,i)}\bigg\}
  \end{align*}
\end{lemma}
\subsection{Proof of Theorem \ref{thm:mpc}}
% We will first need the following statistic associated to the generator $Q$ of a Markov chain.
% \begin{align*}
  % \overline{\mathscr{M}}(Q) &= \min\limits_{\substack{x,y\in\Omega\\x\sim y}}\max\{Q(x,y),Q(y,x)\}\\
% \end{align*}
To prove Theorem \ref{thm:mpc}, we will first need to use the inductive framework to prove some functional-analytic estimates for SCP measures.
% \begin{theorem}\label{thm:univFI}Consider a probability distribution $\pi$ with the SCP and let $Q$ be a reversible Markov generator w.r.t. $\pi$. Then
  % \begin{align*}
    % \lambda(Q)\geq\overline{\mathscr{M}(Q)}
  % \end{align*}
  % where $\lambda(\cdot)$ is the constant in the Matrix valued Poincare inequality of the associated Markov generator.
% \end{theorem}
% \begin{proof}Theorem \ref{thm:univFI} will be proven using induction on the number of states $|\Omega|$. Consider a probability distribution $\pi$ on $\{0,1\}^n$ whose support $\Omega$ has at least two elements. This means that there is at least one index $\ell \in [n]$ such that the subsets
% \begin{align*}
  % \Omega_0 = \{x \in \Omega:x_\ell = 0\} \  \ \ \Omega_1 = \{x \in \Omega:x_\ell = 1\}
% \end{align*}
% are both non-empty. We associate to $Q$,a Markov generator for $\pi$, the restriction and projection chains $(Q_0,\Omega_0,\pi_0), \ (Q_1,\Omega_1,\pi_1)$ and $(\hat{Q},\{ 0,1\},\hat{\pi})$ based on the partition $\Omega = \Omega_0 \cup \Omega_1$. We can write $\hat{Q}$ as
% \begin{align*}
  % \hat{Q}=\left[\begin{array}{cc}-a & a\\b & -b\end{array}\right]\end{align*}
  % For induction, we first need to prove a Matrix valued Poincare inequality for two state chains.
  \begin{lemma}\label{lem:2stPI} Given a generator $\hat{Q}$ for a two state $(\{ 0,1\})$ chain. We have for any $\mathbf{F}:\{ 0,1\}\rightarrow \mathbb{S}_d$,
    \begin{align*}
      (a+b)\mathbf{Var}[\mathbf{F}]\preceq \mathcal{E}_{\hat{Q}}[\mathbf{F},\mathbf{F}]
    \end{align*}
  \end{lemma}
  \begin{proof}
    It is easy to prove that for a Markov chain with generator Q, the stationary distribution $\pi$ satisfies $\pi(1)=\frac{a}{a+b}$ and $\pi(0)=\frac{b}{a+b}$.
    Now we can focus on understanding the Dirichlet form and variance of any matrix valued function $\mathbf{F}$ .
    \begin{align*}
      \mathcal{E}[\mathbf{F},\mathbf{F}] &= \frac{b}{a+b}\cdot a(\mathbf{F}(0)-\mathbf{F}(1))^2\\
      \mathbf{Var}[\mathbf{F}] &= \mathbb{E}[\mathbf{F}^2] - (\mathbb{E}[\mathbf{F}])^2\\
      &=\frac{1}{a+b}(b\mathbf{F}(0)^2 + a\mathbf{F}(1)^2) - \frac{1}{(a+b)^2}(b\mathbf{F}(0) + a\mathbf[F](1))^2\\
      &=\frac{ab}{(a+b)^2}(\mathbf{F}(0)-\mathbf{F}(1))^2
    \end{align*}
    which concludes that the Poincare constant is $(a+b)$
  \end{proof}
We will also need the following lemma which characterizes the quality of the couplings for SCP measures.
\begin{lemma}{[Exploiting the SCP, Lemma 4 in \cite{HS19}]}\label{lem:lem4} If $\pi$ has the SCP, then there is a coupling $\kappa$ of $\pi_0, \pi_1$ supported on
  \begin{align*}
    \{(x,y)\in \Omega_0 \times \Omega_1 : x \sim y\}
  \end{align*}
  In particular, taking $\kappa_{01}$ and $\kappa_{10}$ to be $\kappa$ and its transpose respectively, we obtain
  \begin{align*}
    \chi \geq \frac{\overline{\mathscr{M}}(Q)}{\max\{a,b\}}
  \end{align*}
  where $\overline{\mathscr{M}}(Q) = \min\limits_{\substack{x,y\in\Omega\\x\sim y}}\max\{Q(x,y),Q(y,x)\}$
\end{lemma}

% Finally, we can complete the induction step. Using Lemmas \ref{lem:2stPI} and \ref{lem:lem4}, we get
% \begin{align*}
  % \chi\lambda(\hat{Q})\geq \overline{\mathscr{M}}(Q)
% \end{align*}
% Using Lemma \ref{lem:recPI}, we get
% \begin{align*}
  % \lambda(Q)\geq\min\{\overline{\mathscr{M}}(Q),\lambda(Q_0),\lambda(Q_1)\}
% \end{align*}
% As SCP measures are closed under conditioning, we get that $\pi_0$ and $\pi_1$ are also SCP. Thus, the restriction chains $(\Omega_0,\pi_0,Q_0)$ and $(\Omega_1,\pi_1,Q_1)$ are SCP and we get $\lambda(Q_i)\geq \overline{\mathscr{M}}(Q_i)$ for $i=0,1$ by induction.
% Since we have that $\overline{\mathscr{M}}(Q_i)\geq\overline{\mathscr{M}}(Q)$, we get
% \begin{align*}
  % \lambda(Q)\geq\overline{\mathscr{M}}(Q)
% \end{align*}
% \end{proof}
We will now prove theorem \ref{thm:mpc}.
\begin{proof}{[of Theorem \ref{thm:mpc}]} The proof is unchanged from the proof in \cite{HS19}. We will induct on the number of elements $n$. We will change the normalization of the generator : For any distribution $\pi$ on $\{0,1\}^n$ with the SCP, our goal is to prove the existence of a flip-swap walk $Q$ for $\pi$ such that $\lambda(Q)\geq 1$ and $\Delta(Q)\leq \begin{cases}
n& \text{always}\\
2k&\text{ if } \pi \text{ is k-homogenous}
\end{cases}$
Once we have established this, after dividing $Q$ by $\Delta(Q)$, we will be done.
As before, take a coordinate $\ell\in[n]$ and define the sets $\Omega_0 = \{x\in\Omega : x_\ell =0\}$ and $\Omega_1 = \{x\in\Omega : x_\ell = 1\}$ are both non-empty. We now consider the projection and restriction chains. Let their stationary distributions be $\hat{\pi},\pi_0,\pi_1$ respectively. We can of course view $\pi_0,\pi_1$ as distributions on $\{ 0,1 \}^{n-1} $ and both distributions are SCP. If $\pi$ is $k$-homogenous, then $\pi_0$ is also $k$-homogenous while $\pi_1$ is $(k-1)$-homogenous. Thus, using our induction hypothesis, we can say that there exists a flip-swap walk generator $Q_i$ for $\pi_i$ (for i=0,1) satisfying $\lambda(Q_i)\geq 1$ and $\Delta(Q_i)\leq \begin{cases}
n-1& \text{always}\\
2(k-i)&\text{ if } \pi \text{ is k-homogenous}
\end{cases}$
We also have, from Lemma \ref{lem:lem4}, a coupling $\kappa$ of $\pi_0$ and $\pi_1$ supported on $\{x\in\Omega_0,y\in\Omega_1 : x\sim y\}$.
\end{proof}
Hermon and Salez \cite{HS19} define the following flip-swap walk $Q$ for $\pi$ as follows : for $x\neq y \in\Omega$,
$Q(x,y)=\begin{cases}
Q_0(x,y) & \text{if } (x,y) \in \Omega_0\times\Omega_0\\
Q_1(x,y) & \text{if } (x,y) \in \Omega_1\times\Omega_1\\
\frac{\hat{\pi}(0)\hat{\pi}(1)\kappa(x,y)}{\pi(x)} & \text{if } (x,y) \in \Omega_0\times\Omega_1\\
\frac{\hat{\pi}(0)\hat{\pi}(1)\kappa(y,x)}{\pi(y)} & \text{if } (x,y) \in \Omega_1\times\Omega_0
\end{cases}$
and the diagonal is adjusted to make the row sums 0.
We will now use Lemma \ref{lem:recPI} with $\Omega_0,\Omega_1$ as our partition and the couplings $\kappa_{01}$ being $\kappa$ and $\kappa{10}$ being transpose of $\kappa$. We also have that the projection generator $\hat{Q}$ satisfies $\lambda(\hat{Q})\geq \hat{Q}(0,1)+\hat{Q}(1,0)$ from Lemma \ref{lem:2stPI}. We also have that for any $x\in\Omega_0, y\in\Omega_1$ with $\kappa(x,y)>0$,
\begin{align*}
  \frac{\pi(x)Q(x,y)}{\hat{\pi}(0)\hat{Q}(0,1)\kappa_{01}(x,y)} = \frac{\hat{\pi}(1)}{\hat{Q}(0,1)} =\frac{\hat{\pi}(0)}{\hat{Q}(1,0)} = \frac{\pi(y)Q(y,x)}{\hat{\pi}(1)\hat{Q}(1,0)\kappa_{01}(y,x)}
  \end{align*}
  where the second equality follows from reversibility. Hence, we get
  \begin{align*}
    \chi = \frac{\hat{\pi}(0)}{\hat{Q}(1,0)}=\frac{\hat{\pi}(1)}{\hat{Q}(0,1)}
  \end{align*}
  Finally combining the two inequalities above, we get
  \begin{align*}
    \chi\lambda(\hat{Q})\geq \hat{\pi}(0)+\hat{\pi}(1) = 1
  \end{align*}
  From our induction hypothesis, we had $\lambda(Q_i)\geq 1$, so we can use Lemma \ref{lem:recPI} to get that $\lambda(Q)\geq 1$. We just need to bound the diagonal of $Q$ now. For $x \in \Omega_0$, we have by construction
  \begin{align*}
    -Q(x,x)=\sum\limits_{y \in \Omega_0 \diagdown x}Q_0(x,y) + \frac{\hat{\pi}(0)\hat{\pi}(1)}{\pi(x)}\sum\limits_{y\in\Omega_1}\kappa(x,y)
  \end{align*}
Using the definition of $\Delta(Q_0)$ and the fact that the first marginal of $\kappa$ is $\pi_0$, we deduce that
\begin{align*}
  -Q(x,x) \leq \Delta(Q_0) + \hat{\pi}(1) = \Delta(Q_0)+\mathbb{E}_{\pi}[X_\ell]
\end{align*}
Similarly, if $x\in\Omega_1$, we have
\begin{align*}
  -Q(x,x) \leq \Delta(Q_1)+\hat{\pi}(0)\leq\Delta(Q_1)+x_\ell
\end{align*}
as $x_\ell=1$. So using the induction hypothesis on $\Delta(Q_i)$, we get that for all $x\in\Omega$,
\begin{align*}\Delta(Q)\leq \begin{cases}
n& \text{always}\\
2k+\mathbb{E}_\pi[X_\ell]-x_\ell&\text{ if } \pi \text{ is k-homogenous}
\end{cases}\end{align*}
Finally to bound the second line by $2k$, we average over the coordinate $\ell$. The generator will be written as $Q^{(\ell)}$ to denote the dependence on $\ell$.For each $ell\in[n]$, the construction of our generator produces a flip-swap walk $Q^{(\ell)}$ for $\pi$ which satisfies $\lambda(Q^{(\ell)})\geq 1$ and $\Delta(Q^{(\ell)})\leq n$. Now we crucially notice that the matrix valued variance is unchanged under averaging (because it doesnt depend on the generator) and the Dirchlet form has a linear dependence on the generator $Q^{(\ell)}$. Hence, $\lambda(Q)$ and $\Delta(Q)$ is preserved under convex combinations. So we can define an average generator $Q^*=\frac{1}{n}\sum\limits_{\ell=1}^{n}Q^{(\ell)}$ automatically satisfies $\lambda(Q^*)\geq 1$ and $\Delta(Q^*)\leq n$. If $\pi$ is $k$-homogenous, then for every $x\in\Omega$, we get
\begin{align*}
  -Q^*(x,x)=-\frac{1}{n}\sum\limits_{\ell=1}^{n}Q^{(\ell)}(x,x) \leq \frac{1}{n}\sum\limits_{\ell=1}^{n}(2k + \mathbb{E}[X_\ell] - x_\ell) = 2k
\end{align*}
Thus $Q^{*}$ satisfies our desired properties and we are done.

\section{A Bernstein Inequality from Matrix valued Poincare inequalities}
In this section we show how to use Matrix valued Poincare inequalities to obtain matrix Bernstein style inequalities.

We will need a well-known monotonicity property of the trace exponential \cite{TroppIntro15}
\begin{lemma}
  \label{lem:trMonotone}
  For any monotone function $f$, for any Hermitian matrices $\bf{A}$ and $\bf{H}$ such that $\bf{A} \preceq \bf{H}$, then $\Tr\big[f(\mathbf{A})\big] \leq \Tr\big[ f(\mathbf{H})\big]$
\end{lemma}

% We also need this formula for derivative of matrix powers \cite{TroppSecondOrder18,Bhatia97}
% \begin{fact}[Derivative of a Matrix Power]\label{fact:matPowerDeriv} Let $\bf{A}:\mathbb{R}\rightarrow \mathbb{M}_d$ be a differentiable function. For each integer $p\geq 1$,
%   \begin{align*}
%     \frac{d}{du}\big(\mathbf{A}(u)^p\big) = \sum\limits_{i=0}^{p-1}\mathbf{A}(u)^{i}\frac{d}{du}\mathbf{A}(u)\mathbf{A}(u)^{p-1-i}
%   \end{align*}
%   In particular,
%   \begin{align*}\mathsf{Tr}\bigg[\frac{d}{du}\big(\mathbf{A}(u)^p\big)\bigg] = p\mathsf{Tr}\bigg[\mathbf{A}(u)^{p-1}\frac{d}{du}\mathbf{A}(u)\bigg]\end{align*}
% \end{fact}
% \begin{lemma}[\cite{TroppSecondOrder18}] Suppose that $\mathbf{H}$ and $\mathbf{A}$ are Hermitian matrices of the same size. Let $q$ and $r$ be integers that satisfy $0 \leq q \leq r$. For each real number $s$ in the range $0 \leq s\leq \min\{q,r-q\}$,
%   \begin{align*}
%     \mathsf{Tr}\bigg[\mathbf{H}\mathbf{A}^q \mathbf{H} \mathbf{A}^{r-q}\bigg] \leq \mathsf{Tr}\bigg[\mathbf{H}|\mathbf{A}|^s \mathbf{H} |\mathbf{A}^{r-s}|\bigg]
%   \end{align*}
% \end{lemma}
The next lemma characterizes Jensen's inequality for matrices. The first statement about operator Jensen's inequality for matrices can be found in \cite{TroppIntro15}, while the Jensen inequality for the trace of a convex function is from \cite{HP03}
\begin{lemma}\label{lem:jensen} Given a decomposition of the identity of the form $\sum\limits_{i=1}^{n}\mathbf{K}_i^{*}\mathbf{K}_i = \mathbf{I}$ and an operator convex function $f$, we have for any set of Hermitian matrices $A_1,\ldots,A_n$,
  \begin{align*}
    f\bigg(\sum\limits_{i}^{n}\mathbf{K}_i^{*} \mathbf{A}_i\mathbf{K}_i\bigg) \preceq \sum\limits_{i}^{n}\mathbf{K}_i^{*}f(\mathbf{A}_i)\mathbf{K}_i
  \end{align*}
  For any convex function, we also have
  \begin{align*}
    \mathsf{Tr}\bigg[f\big(\sum\limits_{i}^{n}\mathbf{K}_i^{*}\mathbf{A}_i \mathbf{K}_i\big)\bigg]\leq \mathsf{Tr}\bigg[\sum\limits_{i=1}^{n}\mathbf{K}_i^{*}f(\mathbf({A}_i)\mathbf{K}_i\bigg]
  \end{align*}
\end{lemma}

Finally, we will need the following matrix identities which we will use to establish a mean value theorem for the Dirichlet form.

\begin{fact}{[Eqn 5.35 in \cite{Bhatia07}]}\label{fct:intNorm} Given any positive semi-definite $\mathbf{A},\mathbf{B}$, any Hermitian $\mathbf{X}$ and any unitarily invariant norm $\| \cdot \|$, we have
  \begin{align*}
    \bigg\|\small\intop_{0}^{1}\mathbf{A}^{t}\mathbf{XB}^{1-t}\mathrm{dt}\bigg\| \leq \frac{1}{2}\|\mathbf{AX} + \mathbf{XB}\|
  \end{align*}
\end{fact}

\begin{fact}{[Eqn 6.42 in \cite{Bhatia07}]}\label{fct:duhamel} For any Hermitian $\mathbf{X},\mathbf{Y}$, we have \begin{align*}
  e^\mathbf{X}-e^{\mathbf{Y}} = \intop_{0}^{1}e^{t\mathbf{X}}(\mathbf{X}-\mathbf{Y})e^{(1-t)\mathbf{Y}}\mathrm{dt}
\end{align*}
\end{fact}
\begin{lemma}\label{lem:var} Given random Hermitian matrices X and Y, we have
  \begin{align*}
    \mathsf{Tr}\big[ \mathbb{E}[(e^\mathbf{X}-e^\mathbf{Y})^2]\big]^p\leq \frac{1}{2}\mathbb{E}\big[\|\mathbf{X}-\mathbf{Y}\|^{2p}\mathsf{Tr}[e^{2p\mathbf{X}}+e^{2p\mathbf{Y}}]\big]
  \end{align*}
\end{lemma}
\begin{proof}
  We have
  \begin{align*}
    \mathsf{Tr}\big[ \mathbb{E}[(e^\mathbf{X}-e^\mathbf{Y})^2]\big]^p&\leq \mathbb{E}\big[\mathsf{Tr}\big[ (e^\mathbf{X}-e^\mathbf{Y})^{2p}\big] \ \ \ \text{(using Lemma \ref{lem:jensen})}\\
    &=\mathbb{E}\big[\mathsf{Tr}\big[ \intop_{0}^{1}e^{t\mathbf{X}}(\mathbf{X}-\mathbf{Y})e^{(1-t)\mathbf{Y}}\big]^{2p}\big] \ \ \ \text{(from Fact \ref{fct:duhamel}) }\\
    &=\mathbb{E}\bigg[\bigg\|\intop_{0}^{1}e^{t\mathbf{X}}(\mathbf{X}-\mathbf{Y})e^{(1-t)\mathbf{Y}}\bigg\|_{2p}^{2p}\bigg]&\\
    &\leq \mathbb{E}\bigg[\bigg(\frac{1}{2}\big\|e^{\mathbf{X}}(\mathbf{X-Y})+(\mathbf{X-Y})e^{\mathbf{Y}}\big\|_{2p}\bigg)^{2p}\bigg] \ \ \ \text{(using Fact \ref{fct:intNorm})}\\
    &\leq \mathbb{E}\bigg[\frac{1}{2^{2p}}\bigg(\big\|e^{\mathbf{X}}(\mathbf{X-Y})\big\|_{2p}+\big\|(\mathbf{X-Y})e^{\mathbf{Y}}\big\|_{2p}\bigg)^{2p}\bigg]\ \ \ \text{(using triangle inequality)}\\
    &\leq \mathbb{E}\bigg[\frac{1}{2}\bigg(\big\|e^{\mathbf{X}}(\mathbf{X-Y})\big\|_{2p}^{2p} + \big\|(\mathbf{X-Y})e^{\mathbf{Y}}\big\|_{2p}^{2p} \bigg)\bigg] \ \ \ \text{(using }(a+b)^p \leq 2^{p-1}(a^p + b^p) \ )\\
    &\leq \frac{1}{2}\mathbb{E}\bigg[\|X-Y\|^{2p}\big(\|e^{\mathbf{X}}\|_{2p}^{2p}+\|e^{\mathbf{Y}}\|_{2p}^{2p}\big)\\
    &=\frac{1}{2}\mathbb{E}\bigg[\|X-Y\|^{2p}\mathsf{Tr}\big[e^{2p\mathbf{X}}+e^{2p\mathbf{Y}}\big]\bigg]
  \end{align*}
\end{proof}
We now use this lemma to derive an upper bound on the trace power of the Dirichlet form.
\begin{lemma}\label{lem:dirichlet} Given a reversible Markov Chain with transition matrix $\mathbf{P}\in\mathbb{R}^{\Omega\times \Omega}$ and stationary measure $\mathbf{\pi}\in\mathbb{R}^{\Omega}$, any function $\mathbf{F} : \Omega \rightarrow \mathbb{S}_d$, we have
  \begin{align*}
    \mathsf{Tr}\big[\mathcal{E}(e^\mathbf{F}, e^\mathbf{F})\big]^p \leq v(F)^{2p} \mathsf{Tr}[\mathbb{E}_{x \sim \pi}[e^{2p\mathbf{F(x)}}]]
  \end{align*}
  where $v(F) = \max\limits_{\substack{x,y\in\Omega\\x\sim y}}\|\mathbf{F(x)}-\mathbf{F(y)}\|$
\end{lemma}
\begin{proof}

  Consider the trace power of the Dirchlet form,
  \begin{align*}
    \mathsf{Tr}\big[\mathcal{E}(e^\mathbf{F},e^\mathbf{F})\big]^p &= \mathsf{Tr}\big[\sum\limits_{x \in \Omega} \mathbb{\pi}_x \sum\limits_{y \in \Omega}\mathbf{P}_{xy}\big(e^\mathbf{F(x)}-e^\mathbf{F(y)}\big)^2\big]^p\\
     &\leq \sum\limits_{x\in\Omega}\pi_x\mathsf{Tr}\big[\sum\limits_{y \in \Omega}\mathbf{P}_{xy}\big(e^\mathbf{F(x)}-e^\mathbf{F(y)}\big)^2\big]^p \ \ \text{(using Lemma \ref{lem:jensen})}\\
    &\leq \frac{1}{2}\sum\limits_{x\in\Omega}\pi_x\sum\limits_{y\in\Omega}\mathbf{P}_{xy}\big(\|\mathbf{F(x)}-\mathbf{F(y)}\|^{2p}\mathsf{Tr}\big[e^{2p\mathbf{F(x)}}+e^{2p\mathbf{F(y)}}\big]\big)\big]\ \ \text{(using Lemma \ref{lem:var} with } \mathbf{X} = \mathbf{F(x)} \text{ and } \mathbf{Y} = \mathbf{F(y)}\text{ )}\\
    &\leq \frac{v(F)^{2p}}{2}\sum\limits_{x,y \in \Omega}\pi_x \mathbf{P}_{xy}\mathsf{Tr}[e^{2p\mathbf{F}(x)}+e^{2p\mathbf{F(y)}}]\\
    &\leq v(F)^{2p} \sum\limits_{x\in\Omega}\pi_x \mathsf{Tr}[e^{2p\mathbf{F(x)}}] \ \ \text{(using reversibility of \textbf{P})}\\
    &=v(F)^{2p} \mathsf{Tr}[\mathbf{E}_{x \sim \pi}[e^{2p\mathbf{F(x)}}]]
\end{align*}
\end{proof}
We are finally ready to use our Matrix valued Poincare inequality to prove our matrix concentration inequality. A standard method to prove matrix concentration uses the exponential moment generating function (mgf) which we state next.
\begin{proposition}{[Matrix Laplace Transform Method \cite{TroppIntro15}]}\label{prop:mgf} Let $\mathbf{X} \in \mathbb{H}_d$ be a random matrix with trace mgf $m(\theta) = \mathsf{Tr}[e^{\theta (\mathbf{X}-\mathbb{E}[\mathbf{X}])}]$. For every $t \in \mathbb{R}$,
  \begin{align*}
    \mathsf{Pr}[\|\mathbf{X}-\mathbb{E}[\mathbf{X}]\|\geq t] \leq 2 \inf\limits_{\theta > 0}e^{-\theta t + \log(m(\theta))}
  \end{align*}
\end{proposition}
Hence, to get our concentration inequality, we just need to get a bound on the moment generating function.
\begin{theorem}\label{thm:mgfBound}
Given a reversible Markov Chain with transition matrix $\mathbf{P}\in\mathbb{R}^{\Omega \times \Omega}$ and its stationary distribution $\pi\in \mathbb{R}^{\Omega}$, any function $\mathbf{F} : \Omega \rightarrow \mathbb{S}_d$ such that it satisfies the Matrix valued Poincare inequality with constant $\lambda$, we have
\begin{align*}\mathsf{Pr}_{x \sim \pi}[\|\mathbf{F}(x)-\mathbb{E}_{y \sim \pi}[\mathbf{F(y)}]\|\geq t] \leq 2d e^{-t^2/4(v(\mathbf{F})^2/\lambda + tv(\mathbf{F})/\sqrt{\lambda})}\end{align*}
  where $v(F) = \max\limits_{\substack{x,y\in\Omega\\x\sim y}}\|\mathbf{F(x)}-\mathbf{F(y)}\|$
\end{theorem}

\begin{proof}
  For notational convenience, we will write $\lambda$ as the inverse of the actual $\lambda$ value. Hence, our Poincare inequality now looks like $\mathbf{Var}[\mathbf{F}]\preceq\lambda\mathcal{E}[\mathbf{F},\mathbf{F}]$To use Proposition \ref{prop:mgf}, we need an upper bound on the mgf. The statement we want to prove, via induction, is \begin{equation}\label{eqn:indStmt}
  (1-\lambda v(F)^2 S_k )\mathsf{Tr}\big[\mathbb{E}[e^{\mathbf{F(x)}}]\big]\leq  \mathsf{Tr}[\mathbb{E}[e^{\mathbf{F(x)}/2^k}]]^{2^k}
\end{equation}

where $\lambda v(\mathbf{F})^2 \leq 1$ and $S_k = \sum\limits_{i=1}^{k}1/2^i$ which tends to 1 as $k \rightarrow \infty$.
Assuming the induction hypothesis, consider the RHS where $n=2^k$
\begin{align*}
  \mathsf{Tr}[\mathbb{E}[e^{\mathbf{F(x)}/n}]]^n&=\mathsf{Tr}[\sum\limits_{x \in \Omega}\pi_x (1+\mathbf{F(x)}/n + (\mathbf{F(x)}/n)^2/2!+\ldots)]^n\\
  &=\mathsf{Tr}[(1 + \mathbb{E}[F(x)]/n)^n] \rightarrow \mathsf{Tr}[e^{\mathbb{E}\mathbf[F(x)]}]
\end{align*}
where the second inequality follows as we can ignore the second order terms and beyond for $n\rightarrow \infty$.
Now notice that for any $\beta > 0$, we have $v(\beta \mathbf{F}) = \beta v(\mathbf{F})$. Now taking $\theta \mathbf{F}$ in place of $\mathbf{F}$ and the requirement of $\lambda v(\theta \mathbf{F})^2 \leq 1$ translates to $\theta^2 \lambda v(\mathbf{F})^2 \leq 1$, we plug this into Equation \ref{eqn:indStmt}, to get
\begin{equation}\label{eqn:mgfB}
  \mathsf{Tr}\big[\mathbb{E}_{x \sim \pi}[e^{\theta\mathbf{F(x)}}]\big]\leq \frac{\mathsf{Tr}[e^{\theta\mathbb{E}_{x \sim \pi}[\mathbf{F(x)}]}]}{1-\theta^2\lambda v(\mathbf{F})^2}
\end{equation}
Now, using Equation \ref{eqn:mgfB} with $\mathbf{F}$ being $\mathbf{F}-\mathbb{E}[\mathbf{F}]$ which is zero mean and noticing that $v(\mathbf{F}-\mathbb{E}[\mathbf{F}]) = v(\mathbf{F})$, we get that
\begin{equation}\label{eqn:mgfBo}
  \mathsf{Tr}\big[\mathbb{E}_{x \sim \pi}[e^{\theta\mathbf{F(x)}}]\big]\leq \frac{d}{1-\theta^2\lambda v(\mathbf{F})^2}
\end{equation}
Now, in the RHS for Proposition \ref{prop:mgf}, we have
\begin{align*}
  2\inf\limits_{\theta>0}e^{-\theta t + \log m(\theta)}&\leq 2\inf\limits_{\theta^2 \lambda v(\mathbf{F})^2\leq 1}e^{-\theta t + \log m(\theta)}\\
  &=2d\inf\limits_{\theta^2 \lambda v(\mathbf{F})^2\leq 1}e^{-\theta t - \log(1-\theta^2 \lambda v(\mathbf{F})^2)}\\
  &\leq 2d\inf\limits_{\theta^2 \lambda v(\mathbf{F})^2\leq 1}e^{-\theta t +\lambda v(\mathbf{F})^2\theta^2 + \frac{\lambda^2 v(\mathbf{F})^4\theta^4}{2(1-\lambda v(\mathbf{F})^2\theta^2)}}\\
\end{align*}
Differentiating w.r.t $\theta$ and setting to zero, we get that the probability is upper bounded by $2de^{-t^2/4(\lambda v(\mathbf{F})^2)+t \sqrt{\lambda}v(\mathbf{F}))}$ and we would be done.

Now, it just remains to prove the induction statement. Our proof proceeds by a doubling trick to transform the variance of $e^\mathbf{F}$ to the exponential moment generating function. We first prove the base case. Consider $k=1$ and consider the function $\mathbf{f}=\mathbf{F}/2$. Using the matrix valued Poincare inequality on $e^\mathbf{f}$, we have
\begin{align*}
  &\mathbf{Var}(e^\mathbf{f}) = \lambda \mathcal{E}(e^\mathbf{f},e^{\mathbf{f}})\\
  \implies&\mathbb{E}[e^\mathbf{2f}] \preceq (\mathbb{E}[e^\mathbf{f}])^2 + \mathcal{E}(e^\mathbf{f},e^{\mathbf{f}})\\
  &\mathbb{E}[e^\mathbf{F}] \preceq (\mathbb{E}[e^\mathbf{f}])^2 + \mathcal{E}(e^\mathbf{f},e^{\mathbf{f}})
\end{align*}
Taking trace on both sides, we have
\begin{align*}
  \mathsf{Tr}[\mathbb{E}[e^\mathbf{F}]]&\leq \mathsf{Tr}[(\mathbb{E}[e^{\mathbf{F}/2}])^2] + \mathsf{Tr}[\mathcal{E}(e^{\mathbf{F}/2},e^{\mathbf{F}/2})]\\
  &\leq \mathsf{Tr}[(\mathbb{E}[e^{\mathbf{F}/2}])^2] + v(\mathbf{F}/2)^2\mathsf{Tr}[\mathbb{E}_{x\sim\pi}[e^{\mathbf{F}}]] \ \ \text{(using Lemma \ref{lem:dirichlet})}\\
  &=\mathsf{Tr}[(\mathbb{E}[e^{\mathbf{F}/2}])^2] + \frac{v(\mathbf{F})^2}{4}\mathsf{Tr}[\mathbb{E}_{x\sim\pi}[e^{\mathbf{F}}]]\\
  &\leq \mathsf{Tr}[(\mathbb{E}[e^{\mathbf{F}/2}])^2] + \frac{\lambda v(\mathbf{F})^2 S_1}{2}\mathsf{Tr}[\mathbb{E}_{x\sim\pi}[e^{\mathbf{F}}]] \ \ \text{(using }\lambda \geq 1\text{ and } S_1 = 1/2\text{ )}\\
  \hspace{0.5em}\bigg(1-\frac{\lambda v(\mathbf{F})^2 S_1}{2}\bigg)\mathsf{Tr}[\mathbb{E}[e^\mathbf{F}]] &\leq  \mathsf{Tr}[(\mathbb{E}[e^{\mathbf{F}/2}])^2]
\end{align*}
which establishes our base case.
Now assume the induction hypothesis for k-1 and consider the function $\mathbf{f}=\mathbf{F}/2^k$. Denote $v_k = v(\mathbf{F}/2^{k}) = v(\mathbf{F})/2^k$ Using the matrix valued Poincare inequality on $e^\mathbf{f}$ and taking trace to the power $p=2^{k-1}$, we get,
\begin{align*}
  \mathsf{Tr}\big[\mathbb{E}[e^{\mathbf{F}/{2^{k-1}}}]\big]^{2^{k-1}} &\leq \mathsf{Tr}\big[(\mathbb{E}[e^{\mathbf{F}/2^{k}}])^2 + \lambda\mathcal{E}(e^{\mathbf{F}/2^k},e^{\mathbf{F}/2^k})\big]^{2^{k-1}}\\
  &=\mathsf{Tr}\big[(1-\lambda v_k^2 )\frac{(\mathbb{E}[e^{\mathbf{F}/2^{k}}])^2}{(1-\lambda v_k^2)} + \lambda v_k^2)(\mathcal{E}(e^{\mathbf{F}/2^k},e^{\mathbf{F}/2^k})/v_k^2)\big]^{2^{k-1}}\\
  &\leq \frac{\mathsf{Tr}\big[\mathbb{E}[e^{\mathbf{F}/2^{k}}]\big]^{2^k}}{(1-\lambda v_k)^{2^{k-1}-1}} + \lambda v_k^2\mathsf{Tr}\big[(\mathcal{E}(e^{\mathbf{F}/2^k},e^{\mathbf{F}/2^k})/v_k^2)\big]^{2^{k-1}}\ \ \text{(using Jensen's inequality as }\lambda v_k^2 \leq 1 \text{ )}\\
  &\leq \frac{\mathsf{Tr}\big[\mathbb{E}[e^{\mathbf{F}/2^{k}}]\big]^{2^k}}{(1-\lambda v_k^2)^{2^{k-1}-1}} + \lambda v_k^2\mathsf{Tr}[\mathbb{E}_{x \sim \pi}[e^{\mathbf{F}}]] \ \ \text{(using Lemma \ref{lem:dirichlet})}
\end{align*}
Now using the induction hypothesis, we can lower bound the LHS above to get,
\begin{align*}
  (1-\lambda v(F)^2 S_{k-1})\mathsf{Tr}[\mathbb{E}_{x\sim\pi}[e^\mathbf{F}]] &\leq \frac{\mathsf{Tr}\big[\mathbb{E}[e^{\mathbf{F}/2^{k}}]\big]^{2^k}}{(1-\lambda v_k^2)^{2^{k-1}-1}}+ \lambda v_k^2\mathsf{Tr}[\mathbb{E}_{x \sim \pi}[e^{\mathbf{F}}]]\\
  (1-\lambda v(F)^2 S_{k-1}-\lambda v_k^2)\mathsf{Tr}[\mathbb{E}_{x\sim\pi}[e^\mathbf{F}]]&\leq \frac{\mathsf{Tr}\big[\mathbb{E}[e^{\mathbf{F}/2^{k}}]\big]^{2^k}}{(1-\lambda v_k^2)^{2^{k-1}-1}}\\
  \mathsf{Tr}\big[\mathbb{E}[e^{\mathbf{F}/2^{k}}]\big]^{2^k}&\geq(1-\lambda v(F)^2 S_{k-1}-\lambda v_k^2)(1-\lambda v_k^2)^{2^{k-1}-1}\mathsf{Tr}[\mathbb{E}_{x\sim\pi}[e^\mathbf{F}]]\\
  &\geq(1-\lambda v(F)^2 S_{k-1}-\lambda v_k^2)(1-\lambda (2^{k-1}-1)v_k^2)\mathsf{Tr}[\mathbb{E}_{x\sim\pi}[e^\mathbf{F}]] \\
   &\geq (1-\lambda v(F)^2 S_{k-1}-\lambda 2^{k-1} v_k^2)\mathsf{Tr}[\mathbb{E}_{x\sim\pi}[e^\mathbf{F}]]\\
   &\geq (1-\lambda v(F)^2 S_{k-1}-\lambda 2^{k-1} v(F)^2/2^{2k})\mathsf{Tr}[\mathbb{E}_{x\sim\pi}[e^\mathbf{F}]]\\
   &\geq (1-\lambda v(F)^2 S_{k})\mathsf{Tr}[\mathbb{E}_{x\sim\pi}[e^\mathbf{F}]]
\end{align*}
where the fourth inequality follows from $(1-x)^n\geq(1-nx)$ and the last line follows as $2^{k-1}/2^{2k} = 1/2^{k+1} \leq 1/2^k$
\end{proof}
\begin{proof}{[of Theorem \ref{thm:srBern}]} Since our function $\mathbf{F}$ is L-Lipshitz, we know that $v(\mathbf{F})\leq 2L$ as the Hamming distance between any pair of connected states in the Markov Chain is 2. Hence, using Theorem \ref{thm:mgfBound} and the fact that $\lambda\geq1/{2k}$ (from \ref{thm:mpc}), we are done.
\end{proof}
\section{Comparison with \cite{KyngSong18}}\label{sec:ks}
We first state the main result of \cite{KyngSong18} below.
\begin{theorem}\label{thm:kyngsong} Suppose $(x_1,\ldots,x_n)\in\{0,1\}^n$ is a random subset whose distribution is a $k$-homogenous strong Rayleigh distribution. Given a collection of PSD matrices $A_1,\ldots,A_n\in\mathbb{R}^{d\times d}$ s.t. for all $e \in [n]$ we have $\|A_e\|\leq 1$ and $\|\mathbb{E}[\sum_e x_e A_e]\|\leq\mu$. Then for any $\varepsilon > 0$,
  \begin{align*}
    \mathsf{Pr}[\|\sum_e x_e A_e - \mathbb{E}[\sum_e x_e A_e]\|\geq \varepsilon \mu]\leq de^{-\varepsilon^2\mu/(\log k + \varepsilon) \Theta(1)}
  \end{align*}
\end{theorem}
Compared to our Theorem \ref{thm:srBern}, immediate differences is that our theorem can handle any symmetric matrix valued Lipshitz function rather than just sums of random matrices associated to each element. Furthermore, \cite{KyngSong18} require their matrices to be PSD whereas we just require the matrices to be symmetric. Finally, in the special case of sums of PSD matrices. While Theorem \ref{thm:kyngsong} is concerned with multiplicative errors whereas we are concerned with additive errors, we can take $t=\varepsilon \mu$ and plugging this into Theorem \ref{thm:srBern}, noticing that sums of spectral norm 1 random matrices have $\Theta(1)$ Lipshitz constant and ignoring constants in both bounds we get that our bound is better when
\begin{align*}
  \frac{\varepsilon^2\mu}{\log(k)+\varepsilon}&\leq \frac{\varepsilon^2\mu^2}{k+\varepsilon\mu\sqrt{k}}\\
  k+\varepsilon\mu\sqrt{k}\leq \mu \log k+\varepsilon\mu
\end{align*}
Now assuming $\varepsilon \leq \frac{1}{\sqrt{k}}$, we get that the above equation is approximately
\begin{align*}
  \mu\geq k/(\log(k))
\end{align*}
Also notice that in this case, $\mu\leq k$ but it is conceivable that $\mu = \Theta(k)$ when there is no strong relation between the underlying strong Rayleigh measure and the matrices for each element. So our bound is better when such a condition is satisfied and we are looking for small error of the form $\varepsilon \leq 1/\sqrt{k}$. If one could establish a Matrix valued modified log-Sobolev and a Matrix Herbst argument, the $\sqrt{k}\varepsilon$ term in our bound will not appear which would allow us to get rid of the $\varepsilon \leq 1/\sqrt{k}$ condition and just require $\mu\geq k/\log(k)$. However, we are not aware of any strategy to get multiplicative bounds even in the scalar concentration settings using log-Sobolev or any other functional inequalities.

\section{Conclusion and Open Questions}
In this paper, we build on the work of \cite{CT13} to obtain matrix concentration of Lipshitz matrix-valued functions using functional inequalities without any stringent assumptions on the rotational symmetry of the random matrices. Our strategy seems general enough that it should work for proving Matrix concentration for a large class of interesting Markov chains for which scalar spectral gap inequalities are proved using the chain decomposition idea of \cite{JSTV04}. We also believe our techniques should be extendible to prove similar concentration for strong log-concave polynomials \cite{ALGV19,CGM19} and would hence imply matrix concentration for uniform distributions on any matroid.

There are several open questions left. The most interesting one seems to be whether one can get a version of matrix valued modified log-Sobolev inequality and a matrix Herbst argument. One main obstacle is that the individual terms inside a natural way to define $\mathcal{E}(\mathbf{f},\log(\mathbf{F}))$ which is what shows up for the modified log-sobolev setting doesn't satisfy the operator convexity properties that we crucially exploited here. Another question is whether one can get rid of the extra $\log(k)$ factor in the multiplicative Chernoff bound in \cite{KyngSong18}, thus implying that $\log(n)/\varepsilon^2$ rather than $(\log n)^2/\varepsilon^2$ spanning trees  suffice for obtaining spectral sparsifiers of graphs.  %Is it also possible to use similar ideas as ours to prove log sobolev inequalities for product spaces which have more structure than general Markov semigroups. In particular, is it possible to recover the matrix hypercontractivity statement of Ben Aroya et al. \cite{BARW08}.

% \section*{Acknowledgements}
%
% We thank Rasmus Kyng, Nikhil Srivastava and Paris Syminelakis for some helpful discussions.

\addreferencesection
\bibliographystyle{amsalpha}
\bibliography{references}%,bib/tensor-pca}

\newcommand{\etalchar}[1]{$^{#1}$}
\providecommand{\bysame}{\leavevmode\hbox to3em{\hrulefill}\thinspace}
\providecommand{\MR}{\relax\ifhmode\unskip\space\fi MR }
\providecommand{\MRhref}[2]{%
  \href{http://www.ams.org/mathscinet-getitem?mr=#1}{#2}
}
\providecommand{\href}[2]{#2}
\begin{thebibliography}{MJC{\etalchar{+}}14}

\bibitem[ABY19]{ABY19}
Richard Aoun, Marwa Banna, and Pierre Youssef, \emph{Matrix poincaré
  inequalities and concentration}.

\bibitem[AGR16]{AGR16}
Nima Anari, Shayan~Oveis Gharan, and Alireza Rezaei, \emph{Monte carlo markov
  chain algorithms for sampling strongly rayleigh distributions and
  determinantal point processes}, Proceedings of the 29th Conference on
  Learning Theory, {COLT} 2016, New York, USA, June 23-26, 2016, 2016,
  pp.~103--115.

\bibitem[ALGV19]{ALGV19}
Nima Anari, Kuikui Liu, Shayan~Oveis Gharan, and Cynthia Vinzant,
  \emph{Log-concave polynomials {II:} high-dimensional walks and an {FPRAS} for
  counting bases of a matroid}, Proceedings of the 51st Annual {ACM} {SIGACT}
  Symposium on Theory of Computing, {STOC} 2019, Phoenix, AZ, USA, June 23-26,
  2019., 2019, pp.~1--12.

\bibitem[BBL09]{BBL09}
Julius Borcea, Petter Br{\"a}nd{\'e}n, and Thomas Liggett, \emph{Negative
  dependence and the geometry of polynomials}, Journal of the American
  Mathematical Society \textbf{22} (2009), no.~2, 521--567.

\bibitem[Bha07]{Bhatia07}
Rajendra Bhatia, \emph{Positive definite matrices}, Springer, 2007.

\bibitem[BLM13]{BLM13}
St{\'{e}}phane Boucheron, G{\'{a}}bor Lugosi, and Pascal Massart,
  \emph{Concentration inequalities - {A} nonasymptotic theory of independence},
  Oxford University Press, 2013.

\bibitem[CGM19]{CGM19}
Mary Cryan, Heng Guo, and Giorgos Mousa, \emph{Modified log-sobolev
  inequalities for strongly log-concave distributions}, FOCS (2019).

\bibitem[CH15]{CH15}
Hao-Chung Cheng and Min-Hsiu Hsieh, \emph{Matrix poincare, phi-sobolev
  inequalities, and quantum ensembles}.

\bibitem[CH16]{CH16}
\bysame, \emph{Characterisations of matrix and operator-valued phi-entropies,
  and operator efron-stein inequalities}.

\bibitem[Cha05]{Chaterjee}
Sourav Chaterjee, \emph{Concentration inequalities with exchangeable pairs
  (ph.d. thesis)}.

\bibitem[CHT15]{CHT15}
Hao-Chung Cheng, Min-Hsiu Hsieh, and Marco Tomamichel, \emph{Exponential decay
  of matrix phi-entropies on markov semigroups with applications to dynamical
  evolutions of quantum ensembles}.

\bibitem[CT13]{CT13}
Richard~Y. Chen and Joel~A. Tropp, \emph{Subadditivity of matrix phi-entropy
  and concentration of random matrices}, CoRR \textbf{abs/1308.2952} (2013).

\bibitem[HP03]{HP03}
Frank Hansen and Gert~K. Pedersen, \emph{Jensen’s trace inequality in several
  variables}, ArXiv e-prints 0303060 (2003).

\bibitem[HS19]{HS19}
Jonathan Hermon and Justin Salez, \emph{Modified log-sobolev inequalities for
  strong-rayleigh measures}, ArXiv preprints, 1902.02775 (2019).

\bibitem[JSTV04]{JSTV04}
Mark Jerrum, Jung-Bae Son, Prasad Tetali, and Eric Vigoda, \emph{Elementary
  bounds on poincare and log-sobolev constants for decomposable markov chains},
  Ann. Appl. Probab. (2004).

\bibitem[KS18]{KyngSong18}
Rasmus Kyng and Zhao Song, \emph{A matrix chernoff bound for strongly rayleigh
  distributions and spectral sparsifiers from a few random spanning trees},
  59th {IEEE} Annual Symposium on Foundations of Computer Science, {FOCS} 2018,
  Paris, France, October 7-9, 2018, 2018, pp.~373--384.

\bibitem[Led99]{Led}
Michel Ledoux, \emph{Concentration of measure and logarithmic sobolev
  inequalities}.

\bibitem[LPW09]{LPW}
David~A Levin, Yuval Peres, and Elizabeth~L. Wilmer, \emph{Markov chains and
  mixing times}.

\bibitem[MJC{\etalchar{+}}14]{MJCFT}
Lester Mackey, Michael~I. Jordan, Richard~Y. Chen, Brendan Farrell, and Joel~A.
  Tropp, \emph{Matrix concentration inequalities via the method of exchangeable
  pairs}, Ann. Probab. (2014).

\bibitem[MT06]{MT06}
Ravi Montenegro and Prasad Tetali, \emph{Mathematical aspects of mixing times
  in markov chains}, Foundations and Trends in Theoretical Computer Science
  \textbf{1} (2006), no.~3.

\bibitem[O'D14]{OD}
Ryan O'Donnell, \emph{Analysis of boolean functions}, Cambridge University
  Press, 2014.

\bibitem[PMT16]{PMT}
Daniel Paulin, Lester Mackey, and Joel~A. Tropp, \emph{Efron–stein
  inequalities for random matrices}, Ann. Probab. (2016).

\bibitem[PP14]{PP14}
Robin Pemantle and Yuval Peres, \emph{Concentration of lipschitz functionals of
  determinantal and other strong rayleigh measures}, Combinatorics, Probability
  {\&} Computing \textbf{23} (2014), no.~1, 140--160.

\bibitem[Tro15]{TroppIntro15}
Joel~A. Tropp, \emph{An introduction to matrix concentration inequalities},
  Foundations and Trends in Machine Learning \textbf{8} (2015), no.~1-2,
  1--230.

\end{thebibliography}

% \appendix

% \input{appendix}

%
%\input{correction-lb-evals}

\end{document}